\documentclass[12pt]{article}
\usepackage{graphicx}
\usepackage{amsmath}
\usepackage{amssymb}
\usepackage{theorem}

\numberwithin{equation}{section}

\newtheorem{thm}{Theorem}[section]
\newtheorem{lemma}[thm]{Lemma}
\newtheorem{prop}[thm]{Proposition}
\newtheorem{cor}[thm]{Corollary}
{\theorembodyfont{\rmfamily}

\newtheorem{conj}[thm]{Conjecture}
\newtheorem{rmk}[thm]{Remark}
}

\renewcommand{\thesubsection}{\arabic{section}.\arabic{subsection}}

\newcommand{\qed}{\hfill \mbox{\raggedright \rule{.07in}{.1in}}}
 
\newenvironment{proof}{\vspace{1ex}\noindent{\bf
Proof}\hspace{0.5em}}{\hfill\qed\vspace{1ex}}
\newenvironment{pfof}[1]{\vspace{1ex}\noindent{\bf Proof of
#1}\hspace{0.5em}}{\hfill\qed\vspace{1ex}}

\newcommand{\R}{{\mathbb R}}
\newcommand{\T}{{\mathbb T}}
\newcommand{\Z}{{\mathbb Z}}
\newcommand{\PP}{{\mathbb P}}

\newcommand{\SO}{{\bf SO}}
\newcommand{\Fix}{\operatorname{Fix}}
\newcommand{\diam}{\operatorname{diam}}

\title{Central Limit Theorems and Suppression of Anomalous Diffusion for Systems with Symmetry}

\author{
Georg A. Gottwald\thanks{School of Mathematics and Statistics, University of Sydney, Sydney 2006 NSW, Australia}
\and
Ian Melbourne\thanks{Mathematics Institute, University of Warwick, Coventry, CV4 7AL, UK}
}

\date{27 November 2012; updated 19 February 2015}

\begin{document}

\maketitle

 \begin{abstract}
We give general conditions for the central limit theorem and weak convergence to Brownian motion (the weak invariance principle / functional central limit theorem) to hold for
observables of compact group extensions of nonuniformly expanding maps.
In particular, our results include situations where the central limit theorem
would fail, and anomalous behaviour would prevail, if the compact group were not present.   

This has important consequences for systems with noncompact Euclidean symmetry and provides the rigorous proof for a conjecture made in our paper:
A Huygens principle for diffusion and anomalous diffusion in spatially extended
systems.
{\em Proc.\ Natl.\ Acad.\ Sci. USA} {\bf 110} (2013) 8411--8416.

 \end{abstract}

 \section{Introduction} 
 \label{sec-intro}

It is by now well-understood that statistical limit laws such as the central limit theorem (CLT) and corresponding invariance principles (convergence to Brownian motion) hold for large classes of nonuniformly hyperbolic dynamical systems~\cite{Gouezel04,MN05,Tyran-Kaminska05,Young98,Young99}.
In this setting, summable decay of correlations is sufficient for the central limit theorem to hold.
There are also numerous results for compact group extensions of such maps~\cite{Dolgopyat02},
in particular for equivariant observables which occur naturally in systems with symmetry~\cite{NMA01,FMT03,MN04,MN04b}.

For systems modelled by Young towers with nonsummable decay of correlations~\cite{Young99}, the central limit theorem generally fails for typical H\"older observables.
In certain instances, there is convergence instead to a stable law with nonstandard normalisation $n^s$, $s>\frac12$, see
Gou\"ezel~\cite{Gouezel04}.  (The corresponding invariance principle, namely weak convergence to a stable L\'evy process with superdiffusive growth rate $t^s$, is also valid~\cite{MZapp}.)

Passing to compact group extensions of such systems, in a recent paper~\cite{GM13} we described a dichotomy whereby the superdiffusion persists or is suppressed in favour of normal diffusion.   The dichotomy is characterised by the equivariance properties of the observables.   The arguments in~\cite{GM13} are heuristic, backed up by numerical simulations.    In this paper, we give a proof of the suppression statements.    The work of~\cite{GM13} was motivated by the study of systems with noncompact Euclidean symmetry, see Subsection~\ref{sec-euc}.

Our main theoretical result, Theorem~\ref{thm-CLT}, gives very general conditions under which the CLT and the weak invariance principle (WIP) hold for equivariant observables of compact group extensions of a nonuniformly expanding map $f:X\to X$.   
Essentially, the problem is reduced to proving that a certain derived observable (denoted $V^*$ in the sequel) lies in $L^2$.   The second main contribution of this paper is to
verify this $L^2$ condition in the situation of~\cite{GM13}.

\begin{rmk}
There is a connection between our results and recent work of
Peligrad \& Wu~\cite{PeligradWu10} and
Cohen \& Conze~\cite{CohenConzesub}.
They consider {\em rotated sums} of the form
$\sum_{j=0}^{n-1} e^{ij\theta}v\circ f^j$ and prove central limit theorems
under extremely mild conditions:  it suffices~\cite{CohenConzesub} that 
$f$ is exact and $v\in L^2$.   This corresponds to the case of
circle extensions with
constant cocycle $h\equiv e^{i\theta}$ and observables 
$\phi(x,\psi)=e^{i\psi}v(x)$ in the notation of this paper.
The constancy of $h$ enables the use of Fourier-analytic techniques.
For our results we require much stronger assumptions on the dynamics and
the function $v$, but we do not require that $h$ is constant.
\end{rmk}

First, we focus on the specific case
of Pomeau-Manneville intermittency maps~\cite{PomeauManneville80}, before turning
to a more general class of nonuniformly expanding maps in Subsection~\ref{sec-Intro-NUE}.

\subsection{Intermittency maps}
For ease of exposition,
we consider the family of maps $f:X\to X$, where $X$ is the interval $[0,1]$,
studied by~\cite{LiveraniSaussolVaienti99}.
For $\gamma\ge0$, let 
\begin{align} \label{eq-LSV}
f(x)=\begin{cases}   x(1+2^\gamma x^\gamma), & x\in[0,\frac12] \\
2x-1, & x\in(\frac12,1].
\end{cases}
\end{align}
We are interested in the statistical properties of
compact group extensions of $f$ for $\gamma\in[0,1)$.   

The statistical properties of $f$ itself are well understood.   
There is a unique absolutely continuous ergodic invariant probability measure $\mu$.
If $\gamma=0$, then $f$ is the
doubling map with exponential decay of correlations.  For $\gamma\in(0,1)$,
it is known~\cite{Hu04} that the correlation function
$\rho(n)=\int_X v\,w\circ f^n\,d\mu-\int_X v\,d\mu\,\int_X w\,d\mu$
satisfies $|\rho(n)|\le Cn^{-((1/\gamma)-1)}$ for $v$ H\"older, $w\in L^\infty$, and moreover this is sharp.

In the case of summable decay of correlations, namely $\gamma\in[0,\frac12)$, the following central limit theorem holds for H\"older observables
$v:X\to\R^d$.   Suppose that $\int_X v\,d\mu=0$ and define
$v_n=\sum_{j=0}^{n-1}v\circ f^j$.    Then $n^{-\frac12}v_n\to_d~Y$
where $Y$ is normally distributed with mean $0$ and variance $\sigma^2$ (typically positive).   The case of nonsummable decay of correlations, $\gamma\in[\frac12,1)$, is quite different.
For $\gamma=\frac12$, there is still convergence to a 
normal distribution, but if $v(0)\neq0$ then it is necessary to normalise
by $(n\log n)^\frac12$ instead of $n^\frac12$.
For $\gamma\in(\frac12,1)$ and $v(0)\neq0$, the required
normalisation is $n^\gamma$ and $n^{-\gamma}v_n\to_d Y_\alpha$ where
$Y_\alpha$ is a one-sided stable law of order $\alpha=1/\gamma$.
The results for $\gamma\in[\frac12,1)$ are due to Gou\"ezel~\cite{Gouezel04},
who also showed that the ordinary central limit theorem prevails for observables
$v$ with sufficiently large H\"older exponent when $v(0)=0$.

To summarise,  the CLT holds in the {\em strongly chaotic} case $\gamma\in[0,\frac12)$,
but anomalous (superdiffusive)
scaling rates hold typically in the {\em weakly chaotic} case
$\gamma\in[\frac12,1)$.  However,
the anomalous diffusion is suppressed, and normal diffusion prevails, for
smooth enough observables that vanish at the origin.
In addition, the corresponding WIPs are valid: Dedecker \& Merlev\`ede~\cite{DedeckerMerlevede09}  prove weak convergence to Brownian motion in the cases where~\cite{Gouezel04} proves the CLT, and Melbourne \& Zweim\"uller~\cite{MZapp} prove weak convergence to a
L\'evy process in the cases where~\cite{Gouezel04} obtains a stable law.
 
\paragraph{Compact group extensions of intermittency maps}
Next, we consider the generalisation of these results for compact group extensions and {\em equivariant} observables~\cite{NMA01}.
In certain situations~\cite{GM13} it turns out that the above results in the weakly chaotic
case are reversed, namely that suppression of anomalous diffusion is generic
and anomalous diffusion is the degenerate case.
So far our claims in~\cite{GM13} on anomalous diffusion are conjectural, but we present here 
rigorous results on suppression.

Let $G$ be a compact connected Lie group with Haar measure $\nu$.
Consider the group extension $f_h:X\times G\to X\times G$ given by
$f_h(x,g)=(fx,gh(x))$ where $h:X\to G$ is a H\"older cocycle.
The product measure $m=\mu\times\nu$ is an $f_h$-invariant probability
measure, and is assumed throughout to be ergodic.

\begin{rmk}
Ergodicity of $m$ is typical in the following
strong sense.   The set of H\"older cocycles $h:X\to G$ for which $m$ is not ergodic lies inside
a closed subspace of infinite codimension in the space of all H\"older cocycles~\cite{FMT05}.
\end{rmk}

Let $\R^d$ be a representation of $G$; without loss $G$ acts orthogonally
on $\R^d$.   
We consider equivariant observables $\phi:X\times  G\to\R^d$ of the form 
$\phi(x,g)=g\cdot v(x)$ where $v:X\to\R^d$ is H\"older.    
Let $\phi_n=\sum_{j=0}^{n-1}\phi\circ f_h^j$.  
Throughout, we suppose that $\int_{X\times G} \phi\,dm=0$.

\begin{thm}[CLT]  \label{thm-CLT_PM}
Let $\gamma\in(0,\frac12)$.  Assume that $v:X\to\R^d$, $h:X\to G$ are H\"older.
Then $n^{-\frac12}\phi_n\to_d N(0,\Sigma)$
as $n\to\infty$, where $\Sigma$ is a $d\times d$ covariance matrix
satisfying $g\Sigma=\Sigma g$ for all $g\in G$.
That is,
\[
m((x,g)\in X\times G:n^{-\frac12}\phi_n(x,g)\in I)\to
\int_I \frac{1}{(2\pi)^{k/2}(\det\Sigma)^{1/2}}\exp\{-\frac12 y^T\Sigma^{-1}y\}\,dy,
\]
as $n\to\infty$, for every open rectangle $I\subset\R^d$.
\end{thm}

\begin{rmk}[WIP]
In the situation of Theorem~\ref{thm-CLT_PM}, we also obtain the following weak invariance principle.  Define $W_n(t)=n^{-\frac12}\phi_{nt}$
for $t=0,\frac1n,\frac2n,\ldots$ and linearly interpolate to obtain
$W_n\in C([0,\infty),\R^d)$.   Then $W_n$ converges weakly to $W$
in $C([0,\infty),\R^d)$, denoted $W_n\to_w W$, where $W$ is $d$-dimensional Brownian motion with covariance matrix $\Sigma$.

Equivalently, for any $T>0$, $k\ge1$, and for any continuous function
$\chi:C([0,T],\R^k)\to\R^k$, we have that $\chi(W_n)\to_d\chi(W)$ as ordinary
$\R^k$-valued random variables (so $m(\chi(W_n)\in I)\to \PP(\chi(W)\in I)$
for any open rectangle $I\subset\R^k$).
Taking $T=1$, $k=d$ and $\chi(p)=p(1)$ we recover Theorem~\ref{thm-CLT_PM}, so the CLT is a special case of the WIP.
\end{rmk}

 \begin{rmk} \label{rmk-extra}    (a) The convergence here is in distribution with respect to the 
 probability measure $m$ on $X\times G$.    In fact, it follows from~\cite{Zweimueller07} that we obtain {\em strong} distributional convergence: convergence in distribution
 to $N(0,\Sigma)$
 holds for any probability measure that is absolutely continuous
 with respect to $m$.
The corresponding statement also holds for the WIP.
 
 \vspace{1ex}
 \noindent(b) By~\cite{MN04}, we obtain strong distributional convergence also with respect to the 
 measure $\mu\times\delta_{g_0}$ for $g_0\in G$ fixed.
 
 \vspace{1ex}
\noindent(c) The covariance matrix 
is typically nondegenerate.   (Again, the degenerate situation $\det\Sigma=0$ 
holds only on a closed subspace of infinite codimension in the space
of H\"older functions $v:X\to\R^d$~\cite{NMA01}.)
\end{rmk}

When $\gamma\in[\frac12,1)$ it is necessary to consider the 
values of the cocycle $h$ and the observable $v$ at the neutral fixed point $0$.
Let $\Fix g=\{w\in\R^d:gw=w\}$ for $g\in G$.  We have the orthogonal splitting
$\R^d=\Fix h(0)\oplus (\Fix h(0))^\perp$.

\begin{thm} \label{thm-suppress}
Let $\gamma\in[\frac12,1)$.
Suppose that $v:X\to\R^d$ and $h:X\to G$ are $\eta$-H\"older, where $\eta>\min\{0,\gamma-\frac12\}$.

If $v(0)\in(\Fix h(0))^\perp$, then
$n^{-\frac12}\phi_n\to_d N(0,\Sigma)$ 
where $\Sigma$ is a $d\times d$
covariance matrix satisfying $g\Sigma=\Sigma g$ for all $g\in G$.

Again the convergence is in the sense of strong distribution, $\Sigma$  is typically nondegenerate, and
the corresponding weak invariance principle holds. 
\end{thm}

The heuristic arguments in~\cite{GM13} generalise in the current context to
yield the following conjecture:

\begin{conj} \label{conj-stable_PM}
If $\gamma\in(\frac12,1)$ and $v(0)\not\in(\Fix h(0))^\perp$, then we conjecture that $n^{-\gamma}\phi_n$ converges in distribution to a
$d$-dimensional stable law of order $\alpha=1/\gamma$.

Similarly, if $\gamma=\frac12$ and $v(0)\not\in(\Fix h(0))^\perp$, then we conjecture that
$(n\log n)^{-\frac12}\phi_n$ converges in distribution to a
$d$-dimensional normal distribution.   
\end{conj}

\begin{rmk}    As in~\cite[Section~4(a)]{FMT03}, our set up decomposes naturally into the cases where $G$ acts trivially on $\R^d$ and where $G$ acts fixed-point freely on $\R^d$
(so if $w\in\R^d$ and $g\cdot w=w$ for all $g\in G$, then $w=0$).
In the latter case, the condition
$\int_{X\times  G}\phi\,dm=0$ is automatically satisfied~\cite{NMA01}.
\end{rmk}

\subsection{Extensions of nonuniformly expanding maps}
\label{sec-Intro-NUE}

The intermittency maps~\eqref{eq-LSV} are examples of nonuniformly expanding maps.   This is a large class of dynamical systems that can be modelled by
Young towers~\cite{Young99} and whose statistical properties are well-understood.   The main result of this paper, Theorem~\ref{thm-CLT} below, gives very general conditions under which the CLT and WIP hold for equivariant observables of compact group extensions of such maps.

Let $(X,d)$ be a locally compact separable bounded metric space with
Borel probability measure $\mu_0$ and let $f:X\to X$ be a nonsingular
transformation for which $\mu_0$ is ergodic.
Let $Y\subset X$ be a measurable subset with $\mu_0(Y)>0$, and
let $\alpha$ be an at most countable measurable partition
of $Y$ with $\mu_0(a)>0$ for $a\in\alpha$.    Suppose that there is an $L^1$
{\em return time} function $r:Y\to\Z^+$, constant on each $a\in\alpha$,
and constants $\lambda>1$, $\eta\in(0,1]$, $C\ge1$
such that for each $a\in\alpha$,
\begin{itemize}
\item[(1)] $F=f^{r(a)}:a\to Y$ is a bijection with measurable inverse.
\item[(2)] $d(Fx,Fy)\ge \lambda d(x,y)$ for all $x,y\in a$.
\item[(3)] $d(f^\ell x,f^\ell y)\le Cd(Fx,Fy)$ for all $x,y\in a$,
$0\le \ell <r(a)$.
\item[(4)] $g_a=\frac{d(\mu_0|a\circ F^{-1})}{d\mu_0|_Y}$
satisfies $|\log g_a(x)-\log g_a(y)|\le Cd(x,y)^\eta$ for all
\mbox{$x,y\in Y$}.
\end{itemize}

\begin{rmk} \label{rmk-LSV}
	For the intermittency maps~\eqref{eq-LSV} a natural choice is
$Y=[\frac12,1]$.    Conditions (1)--(4) are valid for all $\gamma\ge0$ and
the condition that $r$ is integrable holds if and only if $\gamma\in[0,1)$.
\end{rmk}

Such a dynamical system $f:X\to X$ is called {\em nonuniformly expanding}.
There is a unique $f$-invariant probability measure $\mu$ on $X$ equivalent
to $\mu_0$ (see for example~\cite[Theorem~1]{Young99}).

As before, we consider compact group extensions
$f_h:X\times G\to X\times G$,
$f_h(x,g)=(fx,gh(x))$.   Again, the invariant product measure $m=\mu\times\nu$
is assumed to be ergodic.
Let $\phi:X\times G\to\R^d$ be an observable of the form $\phi(x,g)=g\cdot v(x)$ 
where $v:X\to\R^d$ and $G$ acts orthogonally on $\R^d$.  

To study the statistical properties of the observable $\phi$, we follow the 
standard approach of {\em inducing} where we pass from the nonuniformly expanding map $f:X\to X$ (and its group extension on $X\times G$) to the uniformly expanding map $F=f^r:Y\to Y$ (and its group extension on $Y\times G$).
There is a trade-off between the improvement of $f$ and the deterioration
of the cocycle $h:X\to G$ and observable $\phi:X\times G\to\R^d$, stemming from the possibility that the return time function $r$ may be large.
Hence it is
necessary to consider an induced cocycle $H:Y\to G$ and 
an induced observable $\Phi:Y\times G\to\R^d$ which incorporate this information
(see Section~\ref{sec-NUE}).   

To state our main result,
it suffices to introduce induced versions of the function $v:X\to\R^d$.
Define $V,\,V^*:Y\to\R^d$,
\[
V(y)=\sum_{j=0}^{r(y)-1}h_j(y)v(f^jy), \quad
V^*(y)=\max_{0\le \ell<r(y)}\Bigl|\sum_{j=0}^{\ell}h_j(y)v(f^jy)\Bigr|,
\]
where $h_j(y)=h(y)\cdots h(f^{j-1}y)$.
Note that if $h$ is measurable and $v\in L^\infty$, then  $r\in L^p$ implies that $V$ and $V^*$ lie in $L^p$.

\begin{thm}  \label{thm-CLT}
Suppose that $f:X\to X$ is nonuniformly expanding and that $r:Y\to\Z^+$ is constant on partition elements.
Suppose further that $v:X\to\R^d$ and $h:X\to G$ are uniformly H\"older.

If $r\in L^p$ for some $p>1$ and $V\in L^2$, then the CLT holds for $\phi$.
If moreover $V^*\in L^2$, then the WIP holds for $\phi$.
In particular, this is the case if $r\in L^2$.

The additional conclusions in Remark~\ref{rmk-extra} are again applicable.
\end{thm}

For the intermittency maps~\eqref{eq-LSV}, it is well-known that $r\in L^2$
if and only if $\gamma\in[0,\frac12)$.  Hence Theorem~\ref{thm-CLT_PM}
is an immediate consequence of Theorem~\ref{thm-CLT}.   
For $\gamma\in[\frac12,1)$ it is still the case that $r\in L^p$ for some
$p>1$, so given Theorem~\ref{thm-CLT} it suffices to verify that $V^*\in L^2$ in order to prove Theorem~\ref{thm-suppress}.

\subsection{Application to Euclidean group extensions}
\label{sec-euc}

The work in this paper was motivated by questions related to Euclidean symmetry that we raised in~\cite{GM13}.  
 In particular, Theorem~\ref{thm-suppress} answers one of the main questions in~\cite{GM13} as we now explain.

We say that $\Gamma$ is a Euclidean-type group if $\Gamma$ is 
a semidirect product $\Gamma=G\ltimes\R^d$ where $G$ is a connected closed
subgroup of $\SO(d)$, the group of $d\times d$ orthogonal matrices.
It is assumed that the group multiplication is given by
$(g_1,p_1)\cdot(g_2,p_2)=(g_1g_2,p_1+g_1p_2)$ where
$g_1p_2$ is matrix multiplication.
When $G=\SO(d)$ this is the $d$-dimensional Euclidean group.

Consider the noncompact group extension $f_\xi:X\times \Gamma\to X\times\Gamma$,
\[
f_\xi(x,\gamma)=(fx,\gamma\xi(x)),
\]
where $\xi:X\to\Gamma$ is a measurable cocycle.   
Write $\gamma=(g,p)$ where $g\in G$, $p\in\R^d$.
Similarly, write $\xi=(h,v)$ where $h:X\to G$, $v:X\to\R^d$.
Then the group extension becomes
\[
f_\xi(x,\gamma)=(fx,gh(x),p+gv(x))=(f_h(x,g),p+\phi(x,g)),
\]
where $\phi(x,g)=gv(x)$.   In particular, as noted in~\cite{NMA01}, 
the noncompact part
of the dynamics is governed by the statistical properties of
the equivariant observable $\phi:X\times G\to\R^d$.

Let $\T$ denote the maximal torus in $G$, with fixed point space
$\Fix(\T)=\{v\in\R^d:gv=v\,\text{for all}\,g\in\T\}$.
Typically $h(0)$ generates a maximal torus.  Hence if 
$\Fix(\T)=\{0\}$, then typically $\Fix(h(0))=\{0\}$  so that the 
hypothesis $v(0)\in(\Fix(h(0))^\perp$ is automatically satisfied.
For $v$, $h$ H\"older, our main results imply that superdiffusion is 
typically suppressed for such Euclidean-type groups.   

On the other hand, if $\Fix(\T)\neq\{0\}$, then typically $v(0)\not\in(\Fix(h(0))^\perp$ and superdiffusive behaviour is conjectured.

In the special case of the Euclidean group $\Gamma=\SO(d)\ltimes\R^d$
we have $\Fix(\T)=\{0\}$ if and only if $d$ is even.   In~\cite{GM13}
we gave heuristic arguments, supported by numerics, for suppression of 
superdiffusion in even dimensions and existence of superdiffusion in odd dimensions.   
The claim that superdiffusion is typically suppressed in even dimensions is
a consequence of Theorem~\ref{thm-suppress}.
The claim about existence of superdiffusion in odd dimensions is
a 
special case of Conjecture~\ref{conj-stable_PM}.

\begin{rmk}
Suppose that $d$ is even.   The action of $G=\SO(d)$ on $\R^d$ is irreducible, so the property $g\Sigma=\Sigma g$ for $g\in G$ implies that $\Sigma=\sigma^2 I_d$ for some $\sigma>0$.   (Typically $\sigma>0$.)   

Similarly, for $d$ odd the conjectured limits in Conjecture~\ref{conj-stable_PM}
are symmetric.
\end{rmk}

\vspace{1ex}
The structure of the remainder of the paper is as follows.
In Section~\ref{sec-GM}, we prove the CLT and WIP for group extensions of a class of 
uniformly expanding maps called Gibbs-Markov maps.
In Section~\ref{sec-NUE}, we use the result in Section~\ref{sec-GM} to prove
Theorem~\ref{thm-CLT}.
In Section~\ref{sec-PM}, we show that Theorem~\ref{thm-suppress} follows from
Theorem~\ref{thm-CLT} by verifying that $V^*\in L^2$.

The argument in Section~\ref{sec-NUE} relies on the method of inducing statistical limit laws, which is by now standard for the CLT.  The corresponding result for the 
WIP is a special case of~\cite{MZapp} (where the focus is on the superdiffusive case) but 
the method simplifies significantly in the situation of this paper.  
Hence we have included the required special case
of~\cite{MZapp} in Appendix~\ref{sec-MZ}.

\paragraph{Notation}
We use ``big O'' and $\ll$ notation interchangeably, writing
$a_n=O(b_n)$ or $a_n\ll b_n$ as $n\to\infty$ if there is a constant
$C>0$ such that $a_n\le Cb_n$ for all $n\ge1$.

\section{Central limit theorems for group extensions of Gibbs-Markov maps}
\label{sec-GM}

Suppose that $(Y,\mu)$ is a probability space, and that
$\alpha$ is a countable measurable partition of $Y$.
Let $F:Y\to Y$ be an ergodic measure-preserving map.
It is assumed that the partition $\alpha$ separates orbits of $F$
and that $F|_a:a\to Y$ is a bijection for each $a\in\alpha$.
If $a_0,\dots,a_{n-1}\in\alpha$, we define the $n$-cylinder
$[a_0,\dots,a_{n-1}]=\cap_{i=0}^{n-1}F^{-i}a_i$.
Fix $\theta\in(0,1)$ and define $d_\theta(x,y)=\theta^{s(x,y)}$
where the {\em separation time} $s(x,y)$ is the greatest integer $n\ge0$
such that $x$ and $y$ lie in the same $n$-cylinder.

An observable $V:Y\to\R^d$ is {\em Lipschitz} if 
$\|V\|_\theta=|V|_\infty+|V|_\theta<\infty$ where
$|V|_\theta=\sup_{x\neq y}|V(x)-V(y)|/d_\theta(x,y)$.  The space
$F_\theta(Y,\R^d)$ of Lipschitz observables is a Banach space.
More generally we say that an observable $V:Y\to\R^d$ is {\em locally
Lipschitz}, $V\in F_\theta^{\rm loc}(Y,\R^d)$,
if $V|_a\in F_\theta(Y,\R^d)$ for each $a\in\alpha$.
Accordingly, we define 
$D_\theta V(a)=\sup_{x,y\in a:x\neq y}|V(x)-V(y)|/d_\theta(x,y)$.  

Define the potential function $p=\log\frac{d\mu}{d\mu\circ F}:Y\to\R$
and assume that 
$p\in F_\theta^{\rm loc}(Y,\R)$ and moreover that $\sup_a D_\theta p(a)<\infty$.
In particular, $F:Y\to Y$ is {\em Gibbs-Markov}~\cite{Aaronson}.

Let $\alpha_n$ denote the partition of $Y$ into $n$-cylinders.
Also let $q=e^p$ and $q_n=q\,q\circ F\cdots q\circ F^{n-1}$.
Gibbs-Markov maps have the property that there 
exists a constant $D>0$ such that for all $n\ge1$, $a\in\alpha_n$ and $y,y'\in a$, 
\begin{align}  \label{eq-GM}
q_n(y)\le D\mu(a), \quad\text{and}\quad
|q_n(y)-q_n(y')|\le D\mu(a)d_\theta(F^ny,F^ny').
\end{align}

Next let $G$ be a compact connected Lie group acting orthogonally on $\R^d$.  
Given a measurable cocycle $H:Y\to G$, we define the 
$G$-extension $F_H:Y\times G\to Y\times G$, $F_H(y,g)=(Fy,gH(y))$
with invariant measure $m=\mu\times\nu$ (recall that $\nu$ is Haar measure on $G$).
The Euclidean metric on $\R^{d\times d}$ restricts to a pseudometric on $G$
and we can speak of locally Lipschitz cocycles $H\in F_\theta^{\rm loc}(Y,G)$.

\begin{thm} \label{thm-CLT_GM}
Let $\theta\in(0,1)$, $\epsilon\in(0,1]$.
Suppose that $F_H:Y\times G\to Y\times G$ is an ergodic $G$-extension
of a Gibbs-Markov map $F:Y\to Y$ by a locally Lipschitz cocycle 
$H\in F_{\theta^{1/\epsilon}}^{\rm loc}(Y,G)$.
Let $V\in L^2\cap F_\theta^{\rm loc}(Y,\R^d)$,
and define the equivariant observable $\Phi(y,g)=g\cdot V(y)$.
Suppose that $\int_{Y\times G}\Phi\,dm=0$ and
define $\Phi_n=\sum_{j=0}^{n-1}\Phi\circ F_H^j$.

Assume that
\begin{itemize}
\item[(i)] $\sum_{a\in\alpha}\mu(a)|1_aV|_\infty<\infty$.
\item[(ii)]
	$\sum_{a\in\alpha}\mu(a)(D_{\theta^{1/\epsilon}}V(a))^\epsilon(1+|1_aV|_\infty)<\infty$.
\item[(iii)]
	$\sum_{a\in\alpha}\mu(a)(D_{\theta^{1/\epsilon}}H(a))^\epsilon(1+|1_aV|_\infty)<\infty$.
\end{itemize}

Then the limit $\Sigma=\lim_{n\to\infty}\frac1n \int_{Y\times G} \Phi_n \Phi_n^T\,dm$
exists, $g\Sigma=\Sigma g$ for all $g\in G$, and
$\frac{1}{\sqrt n}\Phi_n\to_d N(0,\Sigma)$.

Moreover, if we 
define $W_n(t)=n^{-\frac12}\Phi_{nt}$
for $t=0,\frac1n,\frac2n,\ldots$ and linearly interpolate to obtain
$W_n\in C([0,\infty),\R^d)$,   then $W_n\to_wW$
in $C([0,\infty),\R^d)$ where $W$ is $d$-dimensional Brownian motion with covariance matrix $\Sigma$.
\end{thm}

\begin{rmk}   \label{rmk--CLT_GM}
By Cauchy-Schwarz, the regularity hypotheses $V\in L^2$ and conditions (i)--(iii) are satisfied provided (1) $\sum \mu(a)|1_aV|_\infty^2<\infty$,
(2) $\sum_{a\in\alpha}\mu(a)(D_{\theta^{1/\epsilon}}V(a))^{2\epsilon}<\infty$,
and (3) $\sum_{a\in\alpha}\mu(a)(D_{\theta^{1/\epsilon}}H(a))^{2\epsilon}<\infty$.
\end{rmk}

In the remainder of this section, we prove this result.
Let $L$ denote the transfer operator for $F_H:Y\times G\to Y\times G$
(so $\int_{Y\times G}Lv\,w\,dm=\int_{Y\times G}v\,w\circ F_H\,dm$).
Similarly, let $M$ denote the transfer operator for $F:Y\to Y$.
Let $M_H$ denote the twisted transfer operator, $M_HV=M(H^{-1}\cdot V)$.
In the following result (and throughout the paper) $\Phi=g\cdot V$ is shorthand for $\Phi(y,g)=g\cdot V(y)$ and so on.
\begin{prop} \label{prop-L}
Let $V\in L^1(Y,\R^d)$.
If $\Phi=g\cdot V$, then $L\Phi=g\cdot M_HV$.
\end{prop}

\begin{proof}
Let $\langle\,,\,\rangle$ denote a $G$-invariant inner product on $\R^d$.
The operator $L:L^1(Y\times G,\R^d)\to L^1(Y\times G,\R^d)$ is defined by the relation
$\int_{Y\times G}\langle L\Phi,\Psi\rangle\,dm  =
\int_{Y\times G}\langle \Phi,\Psi\circ F_H\rangle\,dm$
for all $\Psi\in L^\infty(Y\times G,\R^d)$.   By the Peter-Weyl theorem
and the orthogonality relations for compact groups~\cite{BtD}, we can suppose without loss that
$\Psi=g\cdot W$, $W\in L^\infty(Y,\R^d)$.  Hence,
\begin{align*}
\int_{Y\times G}\langle L\Phi,\Psi\rangle\,dm & =
\int_{Y\times G}\langle \Phi,\Psi\circ F_H\rangle\,dm=
\int_{Y\times G}\langle g\cdot V,gH\cdot W\circ F\rangle\,dm \\ & =
\int_{Y}\langle  V,H\cdot W\circ F\rangle\,d\mu =
\int_{Y}\langle H^{-1}\cdot V,W\circ F\rangle\,d\mu \\ & =
\int_{Y}\langle M_H V,W\rangle\,d\mu=
\int_{Y\times G}\langle g\cdot M_H V,\Psi\rangle\,dm.
\end{align*}
The result follows.
\end{proof}

The next strange-looking result is surprisingly useful.
\begin{prop}  \label{prop-strange}
Suppose that $x,a,b\ge0$ and $x\le a$, $x\le b$.  Then
$x\le (1+a)b^\epsilon$ for all $\epsilon\in(0,1]$.
If in addition $a\ge1$, then $x\le ab^\epsilon$ for all $\epsilon\in(0,1]$.
\end{prop}

\begin{proof}  If $b\le1$, then $x\le b\le b^\epsilon$.  Hence certainly
$x\le (1+a)b^\epsilon$.
If $b\ge1$, then $x\le a\le 1+a\le (1+a)b^\epsilon$.
The last sentence follows from obvious modifications.
\end{proof}

\begin{lemma} \label{lem-LY}
Let $\theta\in(0,1)$, $\epsilon\in(0,1]$.
Suppose that 
$V\in F_{\theta^{1/\epsilon}}^{\rm loc}(Y,\R^d)$ and
$H\in F_{\theta^{1/\epsilon}}^{\rm loc}(Y,G)$.
\begin{itemize}
\item[(a)]
	If $\sum_{a\in\alpha}\mu(a)(D_{\theta^{1/\epsilon}}H(a))^\epsilon<\infty$, then
the essential spectral radius of $M_H: F_\theta(Y,\R^d)\to F_\theta(Y,\R^d)$ is at most $\theta$.
\item[(b)]
Suppose that 
(i) $\sum_{a\in\alpha}\mu(a)|1_aV|_\infty<\infty$,
(ii)
$\sum_{a\in\alpha}\mu(a)(D_{\theta^{1/\epsilon}}V(a))^\epsilon(1+|1_aV|_\infty)<\infty$,
and (iii)
$\sum_{a\in\alpha}\mu(a)(D_{\theta^{1/\epsilon}}H(a))^\epsilon|1_aV|_\infty<\infty$.
Then $M_HV\in F_\theta(Y,\R^d)$.
\end{itemize}
\end{lemma}

\begin{proof}
We prove part (b) first.
Now $(M_HV)(y)=\sum_{a\in\alpha}q(y_a)H(y_a)^{-1}V(y_a)$
where $y_a$ denotes the unique preimage of $y$ in $a$.
Using~\eqref{eq-GM},
$|(M_HV)(y)|\le\sum_{a\in\alpha} q(y_a)|V(y_a)|\le D\sum_{a\in\alpha}\mu(a)|1_aV|_\infty$.   By (i), $|M_HV|_\infty <\infty$.

Similarly, $|(M_HV)(y)-(M_HV)(y')|\le I+II+III$, where
\begin{align*}
I & =\sum_{a\in\alpha}|q(y_a)-q(y_a')||V(y_a)|, \qquad 
II  =\sum_{a\in\alpha}q(y_a')|H(y_a)-H(y_a')||V(y_a)|,
\end{align*}
\begin{align*}
III & =\sum_{a\in\alpha}q(y_a')|V(y_a)-V(y_a')|. 
\end{align*}

By~\eqref{eq-GM} and property (i),
\begin{align*}
|I| & \le D\sum_{a\in\alpha} \mu(a) d_\theta(y,y')|1_aV|_\infty
=Dd_\theta(y,y')\sum_{a\in\alpha} \mu(a)|1_aV|_\infty \ll  d_\theta(y,y').
\end{align*}

Next, the estimates $|H(y_a)-H(y_a')|\le 2$ and $|H(y_a)-H(y_a')|\le D_{\theta^{1/\epsilon}}H(a)d_{\theta^{1/\epsilon}}(y_a,y_a')$ together imply by
Proposition~\ref{prop-strange} that
$|H(y_a)-H(y_a')|\le 2(D_{\theta^{1/\epsilon}}H(a))^\epsilon d_{\theta^{1/\epsilon}}(y_a,y_a')^\epsilon$.  Moreover, $d_{\theta^{1/\epsilon}}^\epsilon=d_\theta$.
By~\eqref{eq-GM} and property (iii),
\begin{align*}
	|II| & \le 2D\sum_{a\in\alpha} \mu(a) (D_{\theta^{1/\epsilon}}H(a))^\epsilon d_\theta(y_a,y_a')|1_aV|_\infty
	\\ & =2D\theta d_\theta(y,y')\sum_{a\in\alpha} \mu(a) (D_{\theta^{1/\epsilon}}H(a))^\epsilon |1_aV|_\infty \ll d_\theta(y,y').
\end{align*}

Similarly, we have $|V(y_a)-V(y_a')|\le 2|1_aV|_\infty$ and $|V(y_a)-V(y_a')|\le D_{\theta^{1/\epsilon}}V(a)d_{\theta^{1/\epsilon}}(y_a,y_a')$ which together imply by
Proposition~\ref{prop-strange} that
$|V(y_a)-V(y_a')|\le (1+2|1_aV|_\infty)(D_{\theta^{1/\epsilon}}V(a))^\epsilon d_{\theta^{1/\epsilon}}(y_a,y_a')^\epsilon$.
By~\eqref{eq-GM} and property (ii),
\begin{align*}
	|III| & \le D\sum_{a\in\alpha} \mu(a) (1+2|1_aV|_\infty)(D_{\theta^{1/\epsilon}}V(a))^\epsilon d_\theta(y_a,y_a') \ll d_\theta(y,y').
\end{align*}
Hence $\|M_HV\|_\theta=|M_HV|_\infty+|M_HV|_\theta<\infty$ as required.

Next we prove part (a). 
We claim that $\|M_H^nV\|_\theta\le C(|V|_\infty+\theta^n|V|_\theta)$.
Since the unit ball of $F_\theta(Y,\R^d)$ is compact in $L^\infty$, the
result then follows from~\cite{Hennion93}.

It remains to prove the claim.  This is done by
combining an argument in~\cite[Corollary 4.3(a)]{BHM05}
with the method used for term $II$ in part~(b).
Let $V\in F_\theta(Y,\R^d)$.
First, it is standard that $|M_HV|_\infty\le |V|_\infty$ and so
$|M_H^nV|_\infty\le|V|_\infty$.  Also,
$(M_H^nV)(y)=\sum_{a\in\alpha_n}q_n(y_a)H_n(y_a)^{-1}V(y_a)$ where 
$y_a$ denotes the unique preimage of $y$ under $F^n$ in $a$ and
$H_n=H\,H\circ F\cdots H\circ F^{n-1}$.
Hence $|(M_H^nV)(y)-(M_H^nV)(y')|\le I+II+III$ where
\begin{align*}
I & =\sum_{a\in\alpha_n}|q_n(y_a)-q_n(y_a')||V|_\infty, \qquad 
II  =\sum_{a\in\alpha_n}q_n(y_a')||H_n(y_a)-H_n(y_a')||V|_\infty,
\end{align*}
\begin{align*}
III & =\sum_{a\in\alpha_n}q_n(y_a')|V(y_a)-V(y_a')|. 
\end{align*}
Now $I\le D\sum_{a\in\alpha_n}\mu(a)d_\theta(y,y')|V|_\infty
=Dd_\theta(y,y')|V|_\infty$ and
$III\le D\sum_{a\in\alpha_n}\mu(a)|V|_\theta d_\theta(y_a,y_a')
=D\theta^n|V|_\theta d_\theta(y,y')$.
Also,
\[
II\le D\sum_{a\in\alpha_n}\mu(a)|H_n(y_a)-H_n(y_a')||V|_\infty,
\]
and
\begin{align*}
|H_n(y_a)-H_n(y_a')| & \le 
\sum_{j=0}^{n-1}|H(F^jy_a)-H(F^jy_a')|
\le 2\sum_{j=0}^{n-1}(D_{\theta^{1/\epsilon}}H(F^ja))^\epsilon d_\theta(F^j y_a,F^jy_a')
\\ & =2\sum_{j=0}^{n-1}(D_{\theta^{1/\epsilon}}H(F^ja))^\epsilon\theta^{n-j} d_\theta(y,y').
\end{align*}
Hence 
$II \le 2D \sum_{a\in\alpha_n}\mu(a)\sum_{j=0}^{n-1}(D_{\theta^{1/\epsilon}}H(F^ja))^\epsilon\theta^{n-j}  |V|_\infty d_\theta(y,y')$.
Now
\begin{align*}
	& \sum_{a\in\alpha_n}\mu(a)\sum_{j=0}^{n-1}(D_{\theta^{1/\epsilon}}H(F^ja))^\epsilon\theta^{n-j} 
	=\sum_{j=0}^{n-1}\sum_{b\in \alpha_{n-j}}\sum_{a\in\alpha_n,F^ja=b}\mu(a)(D_{\theta^{1/\epsilon}}H(F^ja))^\epsilon\theta^{n-j}  
	\\ & =\sum_{j=0}^{n-1}\theta^{n-j}\sum_{b\in \alpha_{n-j}}
	(D_{\theta^{1/\epsilon}}H(b))^\epsilon\sum_{a\in\alpha_n,F^ja=b}\mu(a)
	=\sum_{j=0}^{n-1}\theta^{n-j}\sum_{b\in \alpha_{n-j}}
	(D_{\theta^{1/\epsilon}}H(b))^\epsilon\mu(b)
	\\ & \le \sum_{j=0}^{n-1}\theta^{n-j}\sum_{a\in \alpha}
	(D_{\theta^{1/\epsilon}}H(a))^\epsilon\mu(a)
	 \le (1-\theta)^{-1}\sum_{a\in \alpha}
	(D_{\theta^{1/\epsilon}}H(a))^\epsilon\mu(a),
\end{align*}
so $II \ll |V|_\infty d_\theta(y,y')$.
The claim follows by combining these estimates.
\end{proof}

\begin{pfof}{Theorem~\ref{thm-CLT_GM}}
The proof largely follows~\cite{FMT03,MN04b}.
Suppose first that $M_H: F_\theta(Y,\R^d)\to F_\theta(Y,\R^d)$ has no eigenvalues on the unit circle.
By Lemma~\ref{lem-LY}(a), there exists $\tau<1$ such that the 
spectrum of $M_H$ lies strictly inside the ball of radius $\tau$.
In particular, there is a constant $C>0$ such that
$\|M_H^n\|\le C\tau^n$.

By Lemma~\ref{lem-LY}(b), $W=M_HV\in F_\theta(Y,\R^d)$.
Define $\chi=\sum_{j=1}^\infty M_H^jV=\sum_{j=0}^\infty M_H^jW$.
Our assumptions guarantee that this series is absolutely
convergent in $F_\theta(Y,\R^d)$ and hence $\chi\in F_\theta(Y,\R^d)$.
Write $V=\hat V+H\cdot (\chi\circ F)-\chi$.   Then $\hat V\in L^2$ (since $\chi\in F_\theta(Y,\R^d)$
and $V\in L^2$).
Moreover,
\[
M_HV=M_H\hat V+\chi-M_H\chi=
M_H\hat V+\sum_{j=1}^\infty M_H^jV-
\sum_{j=2}^\infty M_H^jV=M_H\hat V+M_HV,
\]
and so $M_H\hat V=0$.  At the level of $Y\times G$, we have
\[
\Phi = \hat\Phi +(g\cdot\chi)\circ F_H-g\cdot\chi,
\]
where $\hat\Phi=g\cdot\hat V$ is an $L^2$ observable and $g\cdot\chi$ lies in 
$L^\infty$.   By Proposition~\ref{prop-L}, $L\hat\Phi=0$.
It follows that the sequence $\{\hat\Phi\circ F_H^n;\,n\ge1\}$ defines
a reverse martingale sequence.
Hence we have decomposed $\Phi$ into a (reverse) $L^2$ martingale $\hat\Phi$ 
and an $L^\infty$ coboundary, as in Gordin~\cite{Gordin69}.   
By ergodicity, it follows as usual that we obtain the CLT and WIP for one-dimensional projections, and hence in $\R^d$ by the Cramer-Wold device.  

It remains to remove the assumption about eigenvalues for $M_H$ on the unit circle.    Suppose that there are $k$ such eigenvalues $e^{i\omega_\ell}$,
$\omega_\ell\in[0,2\pi)$, $\ell=1,\dots,k$ (including multiplicities).
Generalised eigenfunctions yield polynomial growth rates under iteration by $M_H$; this is impossible since $M_H$ is a contraction in $L^\infty$.
Hence we can write $V=V_0+\sum_{\ell=1}^kV_\ell$ where
$\|M_H^nV_0\|_\theta\le C\tau^n\|V_0\|_\theta$ and $M_H V_\ell=e^{i\omega_\ell}V_\ell$.
Correspondingly $\Phi=\Phi_0+\sum_{\ell=1}^k \Phi_\ell$ where
$\Phi_\ell=g\cdot V_\ell$, $\ell=0,\dots,k$.
In particular, we obtain the CLT and WIP for $\Phi_0$ by the above argument, while
$L\Phi_\ell=e^{i\omega_\ell}\Phi_\ell$, $\ell=1,\dots k$.

By ergodicity, the eigenvalue $1$ 
for $L$ corresponds to constant eigenfunctions (these only occur if the representation
$\R^d$ of $G$ includes trivial representations).   
Restricting to observables of mean zero removes these eigenfunctions, and then $1$ is
not an eigenvalue.
It follows that $\omega_\ell\in(0,2\pi)$, $\ell=1,\dots,k$.

Next, a simple argument (see~\cite{MN04b}) shows that $\Phi_\ell \circ F_H
= e^{-i\omega_\ell}\Phi_\ell$ for $\ell=1,\dots,k$, so that
$|\sum_{j=1}^n \Phi_\ell\circ F_H^j|_\infty\le 2|e^{i\omega_\ell}-1|^{-1}|\Phi_\ell|_\infty$
which is bounded in $n$.   Hence the result for $\Phi$ follows from the result for
$\Phi_0$.
\end{pfof}

\section{Central limit theorems for group extensions of nonuniformly
expanding maps}
\label{sec-NUE}

Let $f:X\to X$ be a nonuniformly expanding map of a metric space $(X,d)$,
with probability measure $\mu_0$, partition $\alpha$, integrable return time
$r:Y\to\Z^+$, and return map $F=f^r:Y\to Y$, satisfying
conditions (1)--(4) as described in Section~\ref{sec-Intro-NUE}.
There is a unique $F$-invariant probability measure $\mu_Y$ 
absolutely continuous with respect to $\mu_0|_Y$.
(So from now, the probability measure on $Y$ denoted by $\mu$ in Section~\ref{sec-GM}  is denoted $\mu_Y$.)
It is easily verified that
the map $F:Y\to Y$ is Gibbs-Markov on the probability space $(Y,\mu_Y)$ with 
partition $\alpha$ and $\theta=\lambda^{-\eta}$.
A standard elementary argument shows that
there is a constant $C>0$ such that
$d(x,y)\le Cd_\theta(x,y)^{1/\eta}$ for all $x,y\in Y$.

We continue to let $\mu$ denote the $f$-invariant probability measure on $X$ as
described after Remark~\ref{rmk-LSV}.  
The construction of $\mu$ starting from $\mu_Y$ is given explicitly at the beginning of the proof of Theorem~\ref{thm-CLT} at the end of this section.

As usual, we suppose that $h:X\to G$ is a measurable cocycle into a
compact connected Lie group $G$ with Haar measure $\nu$, and 
we define the group extension $f_h:X\times G\to X\times G$,
$f_h(x,g)=(fx,gh(x))$.   The invariant product measure $m=\mu\times\nu$
is assumed to be ergodic.

The proof of Theorem~\ref{thm-CLT}
proceeds by considering the return map for $f_h$ to the set $Y\times G$ and reducing to the set up in Section~\ref{sec-GM}.
The return time $r:Y\times G\to\Z^+$ is simply $r(y,g)=r(y)$.
Define the {\em induced cocycle} $H:Y\to G$, $H=h_r=h(h\circ f)\cdots(h\circ f^{r-1})$.    Then the return map $F_H=(f_h)^r:Y\times G\to Y\times G$ is given by
$F_H(y,g)=(Fy,gH(y))$, again with ergodic measure $m_Y=\mu_Y\times\nu$.

Let $\phi:X\times G\to\R^d$ be an observable of the form $\phi(x,g)=g\cdot v(x)$ 
where $v:X\to\R^d$ and $G$ acts orthogonally on $\R^d$.  
We define the {\em induced observable}
$\Phi(y,g)=\sum_{j=0}^{r(y)-1}\phi\circ f_h^j(y,g)$.
Then $\Phi(y,g)=g\cdot V(y)$ where
$V(y)=\sum_{j=0}^{r(y)-1}h_j(y)v(f^jy)$.
Let $Z_n=\{y\in Y:r(y)=n\}$.

\begin{prop} \label{prop-induced}
Let $p\ge1$.   
Suppose that $v$ and $h$ are $C^\eta$ for some $\eta\in(0,1]$,
with H\"older constants $|v|_\eta$ and $|h|_\eta$.
Let $\theta\in[\lambda^{-\eta},1)$.  Then
\begin{itemize}
\item[(a)]  $|1_{Z_n}V|_\infty \le |v|_\infty n$.
\item[(b)]  If $r\in L^p$, then $V\in L^p$ and moreover 
$\sum_{a\in\alpha}\mu_Y(a)|1_aV|_\infty^p<\infty$.
\item[(c)]  $|1_{Z_n}V|_\theta \ll (|v|_\eta+|v|_\infty|h|_\eta) n^2$.
\item[(d)]  $|1_{Z_n}H|_\theta \ll |h|_\eta n$.
\end{itemize}
\end{prop}

\begin{proof}
Part (a) is immediate.
Note that $\int_Y r^p\,d\mu_Y=\sum_{n=1}^\infty \mu_Y(Z_n)n^p$
implying part~(b).

Next, let $y,y'\in Z_n$.   Then
\begin{align*}
|H(y)-H(y')| & \le \sum_{j=0}^{n-1}|h(f^jy)-h(f^jy')|
\le |h|_\eta\sum_{j=0}^{n-1}d(f^jy,f^jy')^\eta \\ &
\ll n|h|_\eta d(Fy,Fy')^\eta
\ll n|h|_\eta d_{\theta}(y,y'),
\end{align*}
and part (d) follows.

Finally, for $y,y'\in Z_n$,
$|V(y)-V(y')|=|\sum_{j=0}^{n-1} h_j(y)v(f^jy)-\sum_{j=0}^{n-1} h_j(y')v(f^jy')|
\ll |\sum_{j=0}^{n-1} h_j(y)-h_j(y')||v|_\infty + 
n|v|_\eta d_{\theta}(y,y')$.
Moreover, $|\sum_{j=0}^{n-1} h_j(y)-h_j(y')|\le 
\sum_{j=0}^{n-1}\sum_{k=0}^{j-1}|h(f^kx)-h(f^ky')|\ll |h|_\eta n^2d_{\theta}(y,y')$.
Part (c) follows.
\end{proof}

\begin{lemma} \label{lem-CLT}   Under the hypotheses of Theorem~\ref{thm-CLT},
the induced observable $\Phi=g\cdot V:Y\times G\to\R^d$ satisfies the CLT and WIP.
\end{lemma}

\begin{proof} 
We verify the hypotheses of Theorem~\ref{thm-CLT_GM}.
By assumption, $V\in L^2$.
By Proposition~\ref{prop-induced}(b), condition~(i) holds.

Let $\epsilon\in(0,1)$, $\epsilon\le (p-1)/2$.  
Increase $\theta\in(0,1)$ if necessary so 
that $V\in F_{\theta^{1/\epsilon}}^{\rm loc}(Y,\R^d)$ and
$H\in F_{\theta^{1/\epsilon}}^{\rm loc}(Y,G)$.
By Proposition~\ref{prop-induced}(c), 
\[
	\sum_{a\in\alpha}\mu_Y(a)(D_{\theta^{1/\epsilon}}V(a))^\epsilon(1+|1_aV|_\infty)
\ll \sum_{n\ge1}\mu_Y(Z_n)n^{2\epsilon+1}
\le \sum_{n\ge1}\mu_Y(Z_n)n^p=\int_Y r^p\,d\mu_Y<\infty,
\]
verifying condition (ii) of Theorem~\ref{thm-CLT_GM}.
Similarly, condition (iii) follows from Proposition~\ref{prop-induced}(d).
\end{proof}

The next result, which is proved in the Appendix, is a special case of~\cite{MZapp} showing that, under a mild 
condition, to prove the WIP it suffices to prove the WIP for an induced map.

\begin{thm} \label{thm-MZ}
Suppose that $q:\Omega\to \Omega$ is an ergodic measure-preserving transformation
of a probability space $(\Omega,m)$ and $\phi:\Omega\to\R^d$ is an integrable observable
of mean zero.   Let $\Lambda\subset \Omega$ have positive measure and set $m_\Lambda=(m|\Lambda)/m(\Lambda)$.
Let $r:\Lambda\to\Z^+$
be the first return time to $\Lambda$, namely $r(y)=\inf\{n\ge1:q^ny\in \Lambda\}$.   Suppose that $r$ is integrable
and set $\bar r=\int_\Lambda r\,dm_\Lambda$.   

Define the first return map
$Q=q^r:\Lambda\to \Lambda$ and the induced observable $\Phi=\sum_{j=0}^{r-1}\phi\circ q^j:\Lambda\to\R^d$.
Define the Birkhoff sums $\phi_n=\sum_{j=0}^{n-1}\phi\circ q^j$,
$\Phi_n=\sum_{j=0}^{n-1}\Phi\circ Q^j$.
Also define $\Psi=\max_{0\le\ell<r}|\phi_\ell|:\Lambda\to\R$.

Let $w_n(t)=n^{-\frac12}\phi_{nt}$ and
$W_n(t)=n^{-\frac12}\Phi_{nt}$ for $t=0,\frac1n,\frac2n,\ldots$ and 
linearly interpolate to obtain
processes $w_n,W_n\in C([0,\infty),\R^d)$.   

Assume that 
\begin{itemize}
\item[(a)] $W_n\to_w W$ in $C([0,\infty),\R^d)$ on $(\Lambda,m_\Lambda)$
where $W$ is a $d$-dimensional Brownian motion with covariance matrix
$\Sigma$, and
\item[(b)] $n^{-\frac12}\max_{j=0,\dots,n}\Psi\circ Q^j\to0$ in probability  on $(\Lambda,m_\Lambda)$.
\end{itemize}
Then $w_n\to_w \widetilde W$ in $C([0,\infty),\R^d)$ on $(\Omega,m)$ where
$\widetilde W=(\bar r)^{-\frac12} W$ is a $d$-dimensional Brownian motion with covariance matrix $\widetilde \Sigma=(\bar r)^{-1}\Sigma$.

If condition~(b) fails, then we still have the CLT: 
$n^{-\frac12}\phi_n\to_d N(0,\widetilde\Sigma)$.
\end{thm}

\begin{rmk}
Condition~(b) provides control during individual excursions in $\Omega$ from $\Lambda$.
By Corollary~\ref{cor-Psi} in the Appendix, it suffices that $\Psi\in L^2$
(which is certainly the case if $\phi\in L^\infty$ and $r\in L^2$).

The only property of Brownian motion that is used in the proof is that the sample paths are continuous (relaxing this condition is the main point of~\cite{MZapp}).   Also, the argument goes through if the normalisation factor $n^\frac12$ is replaced by a general regularly varying function.

Analogous methods for obtaining the CLT by inducing can be found for example in~\cite{ChazottesGouezel07, Gouezel07, MT04}.   If the CLT is the main goal, then these approaches may be preferable to Theorem~\ref{thm-MZ}.
\end{rmk}

\begin{pfof}{Theorem~\ref{thm-CLT}}
	Since $r:Y\times G\to\Z^+$ is not necessarily the {\em first} return time to $Y\times G$ for $f_h:X\times G\to X\times G$, we cannot directly apply Theorem~\ref{thm-MZ}.
	This is circumvented by using a tower construction to build an extension of $X\times G$ for which $r:Y\times G\to \Z^+$ is the first return time.

	First we recall the definition of the tower for $f:X\to X$.
	(In doing so, we specify how $\mu$ is constructed from $\mu_Y$.)
	Define the tower map $f_\Delta:\Delta\to\Delta$ by
	$\Delta=\{(y,\ell)\in Y\times \Z:0\le\ell<r(y)\}$ and
	$f_\Delta(y,\ell)=\begin{cases} (y,\ell+1), & \ell\le r(y)-2 \\
(Fy,0), & \ell = r(y)-1 \end{cases}$.
	The probability measure
	$\mu_\Delta=\mu_Y\times{\rm counting}/\int_Y r\,d\mu_Y$ is $f_\Delta$-invariant.
	The projection $p:\Delta\to X$, $p(y,\ell)=f^\ell y$, defines a semiconjugacy between $f_\Delta$ and $f$, and $\mu$ is defined to be
$\mu=p_*\mu_\Delta$.

Similarly, we define the tower map $f_{\Delta\times G}:\Delta\times G\to\Delta\times G$ by
	$\Delta\times G=\{(y,g,\ell)\in Y\times G\times\Z:0\le\ell<r(y)\}$ and
	$f_{\Delta\times G}(y,g,\ell)=\begin{cases} (y,g,\ell+1), & \ell\le r(y)-2 \\
(F_H(y,g),0), & \ell = r(y)-1 \end{cases}$.
	The probability measure
	$m_\Delta=\mu_\Delta\times\nu$ is $f_{\Delta\times G}$-invariant.
	The projection $\pi:\Delta\times G\to X\times G$, $\pi(y,g,\ell)=T_h^\ell(y,g)$, defines a semiconjugacy between $f_{\Delta\times G}$ and $f_h$.
	Moreover, $m=\mu\times\nu$ satisfies
	$m=\pi_*m_\Delta$.

	Starting with the original observable  $\phi=g\cdot v:X\times G\to\R^d$, we define
	$\hat v=v\circ p:\Delta\to\R^d$ and
	$\hat \phi=\phi\circ\pi=g\cdot \hat v:\Delta\times G\to\R^d$. Since $m=\pi_*m_\Delta$, it follows that $\{\hat\phi\circ f_{\Delta\times G}:\,j\ge0\}=_d
	\{\phi\circ f_h:\,j\ge0\}$.
Hence to prove the CLT/WIP for $\phi$ on $X\times G$, it suffices to prove the
CLT/WIP for $\hat\phi$ on $\Delta\times G$.

Since $r:Y\times G\to\Z^+$ is the first return time for $f_{\Delta\times G}:\Delta\times G\to\Delta\times G$, with first return map $F_H:Y\times G\to Y\times G$,
we are now in a position to apply Theorem~\ref{thm-MZ}.
Take $q=f_{\Delta\times G}$, $Q=F_H$, $\Omega=\Delta\times G$, $\Lambda=Y\times G$.
Also, we have $\hat\phi(x,g)=g\cdot \hat v(x)$ and $\Phi(y,g)
=\sum_{j=0}^{r(y)-1}\hat\phi\circ q^j$.  It follows from the definitions
that 
\[
	\Phi=g\cdot V, \quad V(y) =\sum_{j=0}^{r(y)-1}h_j(y)v(f^jy).
\]

Assumption Theorem~\ref{thm-MZ}(a) is immediate from Lemma~\ref{lem-CLT}.
The CLT for $\hat\phi$ follows by the last statement of Theorem~\ref{thm-MZ}.

Next, $\hat\phi_\ell(y,g,0)=\phi_\ell(y,g)$ and
\[
|\phi_\ell(y,g)|= \Bigl|g\cdot\sum_{j=0}^{\ell}h_j(y)v(f^jy)\Bigr|
=\Bigl|\sum_{j=0}^{\ell}h_j(y)v(f^jy)\Bigr|,
\]
so $\Psi(y,g)=V^*(y)$.
Hence the assumption that $V^*\in L^2$ in Theorem~\ref{thm-CLT} implies
that $\Psi\in L^2$ and so
assumption Theorem~\ref{thm-MZ}(b) is satisfied.
\end{pfof}

\section{Central limit theorems for group extensions of intermittency maps}
\label{sec-PM}

In the case of the intermittency maps~\eqref{eq-LSV},
it is well-known that there is a constant $c=c_\gamma>0$
such that $\mu(r>n)\sim cn^{-1/\gamma}$.
In particular, $r\in L^2$ if and only if $\gamma<\frac12$.
Hence, Theorem~\ref{thm-CLT} applies immediately when $\gamma\in[0,\frac12)$.

For $\gamma\in[\frac12,1)$, we have $r\in L^p$ where $p\in(1,2)$.   The CLT and WIP
still hold provided we can verify that $V^*\in L^2$.
Here we require further more specific information about the maps~\eqref{eq-LSV}.
Let $Z_n=\{y\in Y:r(y)=n\}$.
Then it is well known that in fact
$\mu(Z_n)\ll n^{-(1+1/\gamma)}$.
Furthermore, $\diam(f^kZ_n)\ll (n-k)^{-(1+1/\gamma)}$ 
and $|f^ky|\ll (n-k)^{-1/\gamma}$ for $y\in Z_n$, $k=1,\dots,n$.
(See for example~\cite{LiveraniSaussolVaienti99,Gouezel04,Sarig02}.)

\begin{thm}    \label{thm-L2}
Suppose that $f$ is one of the maps~\eqref{eq-LSV}.
Suppose that $v,h\in C^{\eta}$, $\eta\in(0,1]$.
Suppose further that $\eta>\gamma-\frac12$.
If $v(0)\in(\Fix h(0))^\perp$, then $V^*\in L^2$
and  hence $\phi$ satisfies the CLT and WIP.
\end{thm}

\begin{proof}
	We may suppose without loss that $\eta\in(\gamma-\frac12,\gamma)$.

Writing $v=(v-v(0))+v(0)$, we may consider the cases $v(0)=0$ and $v\equiv v(0)$ separately.
The case $v(0)=0$
is identical to the
argument in~\cite{Gouezel04} and is repeated here for completeness.
For $y\in Z_n$,
\begin{align*}
|V^*(y)| & \le \sum_{j=0}^{n-1}|h_j(y)v(f^jy)|\le
\sum_{j=0}^{n-1}|v|_{\eta}|f^jy|^\eta\ll
\sum_{j=1}^{n-1}(n-j)^{-\eta/\gamma}\ll n^{1-\eta/\gamma}.
\end{align*}
Hence,
\begin{align} \label{eq-V*}
\int_Y|V^*(y)|^2\,d\mu
\ll \sum_{n} n^{2-2\eta/\gamma}n^{-(1+1/\gamma)}<\infty.
\end{align}

It remains to consider the case $v\equiv v(0)$.
Since $G$ acts orthogonally and $v(0)\in(\Fix h(0))^\perp$, 
\begin{align} \label{eq-h0}
\sup_{\ell\ge0}\Bigl|\sum_{j=0}^{\ell} [h(0)]^jv(0)\Bigr|<\infty.
\end{align}

Set $\bar h(y)=h(y)h(0)^{-1}$ and
$A_k=h(0)^k\,(\bar h\circ f^k)\, h(0)^{-k}$.
In the noncommutative products below, we write $\prod_{k=0}^{j-1}a_k=a_0a_1\cdots a_{j-1}$.
Then for $j\ge2$,
\begin{align*}
h_j(y) & =
\prod_{k=0}^{j-1}h(f^ky) =\prod_{k=0}^{j-1}[\bar h(f^ky)h(0)]
=\Bigl[\prod_{k=0}^{j-1}A_k(y)\Bigr]h(0)^j \\
& =\sum_{k=1}^{j-1}\Bigl[\prod_{i=0}^{k-1}A_i(y)\Bigr](A_k(y)-I)h(0)^j + A_0(y)h(0)^j.
\end{align*}
Moreover, for $y\in Z_n$,
\[
|A_k(y)-I|=|\bar h(f^ky)-\bar h(0)| =O(|f^ky|^\eta)= O((n-k)^{-\eta/\gamma}).
\]
Hence by~\eqref{eq-h0}, for $\ell\le n$,
\begin{align*}
\Bigl|\sum_{j=0}^\ell h_j(y)v(0)\Bigr| & =
\Bigl|\sum_{j=2}^\ell 
\sum_{k=1}^{j-1}\Bigl[\prod_{i=0}^{k-1}A_i(y)\Bigr](A_k(y)-I)h(0)^jv(0)\Bigr| + O(1) \\
& = \Bigl|\sum_{k=1}^{\ell-1}\Bigl[\prod_{i=0}^{k-1}A_i(y)\Bigr](A_k(y)-I)\sum_{j>k} h(0)^jv(0)\Bigr|+O(1) \\
& \ll \sum_{k=1}^{\ell-1}|A_k(y)-I|+1
 \ll \sum_{k=1}^{\ell-1} (n-k)^{-\eta/\gamma}+1 \\ &
 \le \sum_{k=1}^{n-1} (n-k)^{-\eta/\gamma}+1 \ll n^{1-\eta/\gamma}.
\end{align*}
It follows that
\begin{align*}
|V^*(y)| & =\max_{0\le\ell<n-1}\Bigl|\sum_{j=0}^{\ell}h_j(y)v(0)\Bigr|
\ll n^{1-\eta/\gamma},
\end{align*}
establishing the required estimate just as in~\eqref{eq-V*}.
\end{proof}

In particular, Theorem~\ref{thm-suppress} holds for $v$, $h$ sufficiently H\"older.

\begin{rmk}  \label{rmk-direct}
The resummation argument in the proof of Theorem~\ref{thm-L2} is required in order to fully exploit~\eqref{eq-h0}.   The more direct estimate
$|V^*(y)|\le \sum_{j=0}^n|h_j(y)-h_j(0)||v(0)|+|\sum_{j=0}^n h(0)^jv(0)|\ll n^{2-\eta/\gamma}$ establishes that 
$V\in L^2$ provided $\eta>2\gamma-\frac12$.    
However, even for $h$ Lipschitz ($\eta=1$) this approach succeeds only  for $\gamma<\frac34$.

Our original version of this resummation argument led to the same result but under the unnecessarily stringent restriction $\eta>\gamma$.   The improved (and simplified) argument was pointed out to us by S\'ebastien Gou\"ezel.
\end{rmk}

 \appendix

\renewcommand{\thesubsection}{\Alph{section}.\arabic{subsection}}

\section{Inducing the weak invariance principle}
\label{sec-MZ}

Theorem~\ref{thm-MZ} is a special case of~\cite[Theorem~2.2]{MZapp}.
Since the proof is greatly simplified, and since the published version of~\cite{MZapp} refers to this appendix, we provide the full details here.

As in the proof of Theorem~\ref{thm-CLT_GM}, by the Cramer-Wold device we may suppose without loss that $d=1$.

It is convenient to work throughout with
the Skorohod spaces $\mathcal{D}[0,T]$ and $\mathcal{D}[0,\infty)$ of real-valued cadlag
functions (right-continuous $g(t^{+})=g(t)$ with left-hand limits $g(t^{-})$)
on the respective interval, with the sup-norm topology in the case of
$\mathcal{D}[0,T]$ and the topology of uniform convergence on compact subsets in the case of $\mathcal{D}[0,\infty)$.
(We could equally work with the spaces of continuous functions (replacing certain piecewise constant functions by the piecewise linear continuous interpolants throughout.)

Let $(\Omega,m,q)$, $(\Lambda,m_\Lambda,Q)$ and $r:\Lambda\to\Z^+$ be as in Theorem~\ref{thm-MZ}.  Recall the relation $Q=q^r$ and the notation
$\bar r=\int_\Lambda r\,dm_\Lambda$.
Define the Birkhoff sums $\phi_n=\sum_{j=0}^{n-1}\phi\circ q^j$,
$\Phi_n=\sum_{j=0}^{n-1}\Phi\circ Q^j$,
$r_n=\sum_{j=0}^{n-1}r\circ Q^j$.
Also define $\Psi=\max_{0\le\ell<r}|\phi_\ell|:\Lambda\to\R$
and the cadlag processes $w_n,W_n$, setting
\[
w_n(t)=n^{-\frac12}\phi_{[nt]}, \quad
W_n(t)=n^{-\frac12}\Phi_{[nt]}.
\]

Let $N_k=\sum_{\ell=1}^k1_{\Lambda}\circ
q^\ell=\max\{n\geq0:r_n\leq k\}$ denote the lap numbers of $\Lambda$. The
visits to $\Lambda$, as counted by the lap numbers $N_k$, separate the consecutive \emph{excursions
from} $\Lambda$.
Then we can write
\[
\phi_k=\Phi_{N_k}+R_k\quad\text{on }\Lambda
\]
with remainder term $R_k=\sum_{\ell=r_{N_k}}^{k-1}\phi\circ q^\ell
=\phi_{k-r_{N_k}}\circ Q^{N_k}$ encoding the contribution of the
incomplete last excursion (if any). 
Next, decompose the rescaled process $w_n(t)=n^{-\frac12}\phi_{[nt]}$ 
accordingly, writing
\begin{align*}
w_n(t)=U_n(t)+V_n(t) \quad\text{on }\Lambda
\end{align*}
where
\[
U_n(t)=n^{-\frac12}\Phi_{N_{[tn]}}, \enspace
\text{and}\enspace V_n(t)=n^{-\frac12}
R_{_{[tn]}}. 
\]
The excursions correspond to the intervals $[t_{n,j},t_{n,j+1})$, $j\geq0$, where
$t_{n,j}:\Lambda\to\lbrack0,\infty)$\ is given by $t_{n,j}=r_j/n$. Note
that 
\begin{equation}
t\in\lbrack t_{n,N_{[tn]}},t_{n,N_{[tn]}+1})\text{\quad for }t>0\text{ and }n\geq1.
\label{eq-tn}
\end{equation}

\paragraph{Some almost sure results}
We record some consequences of the ergodic theorem.
But first an elementary observation, the proof of which we omit.

\begin{prop}
\label{prop-cn}Let $s>0$ and let $(c_n)_{n\geq1}$ be a sequence in
$\mathbb{R}$ such that $n^{-s}c_n\to c$. Define a
sequence of functions $C_n:[0,\infty)\to\mathbb{R}$ by letting
$C_n(t)=n^{-s}c_{[tn]}-t^sc$. Then, for any $T>0$,
$(C_n)_{n\geq1}$ converges to $0$ uniformly on $[0,T]$. \qed
\end{prop}

\begin{cor}
\label{cor-Psi}
If $\Psi\in L^2$, then condition~(b) of Theorem~\ref{thm-MZ} is satisfied.
\end{cor}

\begin{proof}
By the ergodic theorem, $n^{-1}\sum_{j=0}^{n-1}\Psi^2\circ Q^j\to \int_\Lambda\Psi^2\,dm_\Lambda$
almost everywhere on $\Lambda$ and hence $n^{-1}\Psi^2\circ Q^n\to0$ almost everywhere.   Now take square roots and apply Proposition~\ref{prop-cn} with $s=\frac12$.
\end{proof}

\begin{lemma}
\label{lem-lap} The lap numbers $N_k$ satisfy
$k^{-1}N_k\to 1/\bar r$ a.e.\ on $\Omega$
as $k\to\infty$.
Moreover, for any $T>0$,
\[
\sup\nolimits_{t\in\lbrack0,T]}| k^{-1}N_{[tk]}-t/\bar r| \to0\text{\quad a.e. on
}\Omega\text{\quad as }k\to\infty.
\]
\end{lemma}

\begin{proof}
Recall that $m(\Lambda)=1/\bar r$ (Kac' formula).
Hence the first statement is immediate from the ergodic theorem. The second then
follows by Proposition~\ref{prop-cn} with $s=1$.
\end{proof}

\paragraph{Convergence of $U_n$.}

We require a standard but technical result.
\begin{prop} \label{prop-Un}  Suppose that $A_n$, $B_n$ are sequences in $\mathcal{D}[0,T]$ and that $A_n\to_w A$, $B_n\to_w B$ in the sup-norm topology.  Suppose further that $A$, $B$ are continuous and that $B$ is nonrandom.
	Then $(A_n,B_n)\to_w(A,B)$ in $\mathcal{D}[0,T]\times\mathcal{D}[0,T]$ with the sup-norm topology.
	\end{prop}

\begin{proof} The issue here is that the sup-norm topology is not separable.
	But since $A,B$ are continuous, convergence in the sup-norm topology is equivalent to convergence in the Skorokhod topology which is separable, and then the result is standard.
\end{proof}

\begin{lemma}
\label{lem-Un}
$U_n\to_w \widetilde W$
in $\mathcal{D}[0,\infty)$ on $(\Lambda,m_\Lambda)$.
\end{lemma}

\begin{proof}
For
$n\geq1$ and $t\in\lbrack0,\infty)$ we define non-negative random variables on
$(\Lambda,m_{\Lambda})$ by letting $u_n(t)=n^{-1}N_{[tn]}$. Since $[u_n(t)\,n]
=N_{[tn]}$, we have
\begin{align*}
U_n(t)=W_n(u_n(t))\quad\text{on }\Lambda\text{ for }n\geq1\text{ and }
t\geq0.
\end{align*}
We regard $U_n$, $W_n$, $W$, $u_n$ as random elements of
$\mathcal{D}=\mathcal{D}[0,\infty)$. 
Let $u$ denote the constant random element of
$\mathcal{D}$ given by $u(t)\equiv t/\bar r$, $t\geq0$.

By Lemma~\ref{lem-lap}, for almost every $y\in
\Lambda$ we have $u_n(.)(y)\to u(.)(y)$ uniformly on compact subsets of
$[0,\infty)$. Hence, $u_n\to u$ a.e. in $\mathcal{D}$.
By condition~(a) of Theorem~\ref{thm-MZ}, we also have $W_n
\to_w W$ 
in $\mathcal{D}$.  Since $W$ and $u$ are continuous and $u$ is nonrandom,  it follows from Proposition~\ref{prop-Un} that
\begin{align*}
(W_n,u_n)\to_w (W,u)\text{\quad
in }\mathcal{D}\times\mathcal{D}.
\end{align*}
The composition map 
$\mathcal{D}\times\mathcal{D}\to\mathcal{D}$,
$(g,v)\mapsto g\circ v$, is well-defined and is easily seen to be
continuous in the sup-norm topology.
Hence it follows from the continuous mapping theorem that
$U_n=W_n\circ u_n\to_w W\circ u=\widetilde W$
as required.
\end{proof}

\paragraph{Convergence of $w_n$.}

\begin{lemma} \label{lem-Wn}
$\|w_n-U_n\|_\infty \le n^{-1/2}\max_{0\le j\le [Tn]}\Psi\circ Q^j$
a.e.\ on $\Lambda$.
\end{lemma}

\begin{proof}
	We decompose $[0,T]$ according to the consecutive excursions, letting $T_j=t_{n,j}\wedge T$, $j\le N_{[Tn]}+1$.
	Then $\|w_n-U_n\|_\infty \le \max_{1\le j\le N_{[Tn]}+1}
\sup_{t\in[T_{j-1},T_j]}|w_n(t)-U_n(t)|$.
But $U_n(t)=w_n(T_{j-1})$ for $t\in[T_{j-1},T_j)$ and
$U_n(T_j)=w_n(T_j)$  so
\begin{align*}
\sup_{t\in[T_{j-1},T_j]}|w_n(t)-U_n(t)| & \le 
\sup_{t\in[T_{j-1},T_j]}|w_n(t)-w_n(T_{j-1})|
\\ & \le n^{-1/2}\max_{0\le\ell<r\circ Q^{j-1}}|\phi_\ell\circ Q^{j-1}|
= n^{-1/2} \Psi\circ Q^{j-1}.
\end{align*}
Since $N_{[Tn]}\le [Tn]$, this yields the required result.
\end{proof}

\begin{pfof}{Theorem~\ref{thm-MZ}}
	First we prove the WIP assuming conditions~(a) and~(b).
Fix any $T>0$.  
It suffices to prove that $w_n\to_w \widetilde W$ in
$\mathcal{D}[0,T]$ on $(\Omega,m)$.
Moreover, since $m_\Lambda$  viewed as a probability measure on $\Omega$ is absolutely continuous with respect to $m$,
it suffices by~\cite[Corollary~3]{Zweimueller07}
to prove that $w_n\to_w \widetilde W$ in
$\mathcal{D}[0,T]$ on $(\Lambda,m_\Lambda)$.

By assumption (b) of Theorem~\ref{thm-MZ}
 and Lemma~\ref{lem-Wn},
$\|w_n-U_n\|_\infty\to0$ in probability on $(\Lambda,m_\lambda)$.
Also, by Lemma~\ref{lem-Un}, $U_n\to_w \widetilde W$ in
$\mathcal{D}[0,T]$ on $(\Lambda,m_\Lambda)$.
Hence, by~\cite[Theorem~3.1]{Billingsley1999}, 
$w_n\to_w \widetilde W$ in
$\mathcal{D}[0,T]$ on $(\Lambda,m_\Lambda)$ as required.

Finally, we prove the CLT assuming only condition~(a).  
By the above, it suffices to prove that $w_n(1)-U_n(1)\to0$ in probability
on $(\Lambda,m_\lambda)$.
A simple argument for this is given in~\cite[Appenxdix~A]{Gouezel07}; we 
sketch the main steps.  First, we can pass to the natural extension so that $q$ is invertible.  Then $w_n(1)-U_n(1)=n^{-1/2}H\circ q^n$ where $H:\Omega\to\R$ is the measurable function given by 
$H(x)=\sum_{j=1}^{s(x)}\phi(q^{-j}x)$ and $s(x)\ge0$ is least such that $q^{-s(x)}x\in\Lambda$.  It follows from invariance of $m$ that $n^{-1/2}H\circ q^n\to0$ in probability
on $(\Omega,m)$ and hence on $(\Lambda,m_\lambda)$.
\end{pfof}

 \paragraph{Acknowledgements}
GAG acknowledges funding from the Australian Research Council.
The research of IM was supported in part by EPSRC Grant EP/F031807/1 held 
at the University of Surrey and in part by a European
  Advanced Grant StochExtHomog (ERC AdG 320977).
We are grateful to S\'ebastien Gou\"ezel for pointing out an improvement to
Theorem~\ref{thm-L2}.

\end{document}